\documentclass{amsart}
\RequirePackage[utf8]{inputenc}
\usepackage{amsmath,amsthm,amsfonts,amssymb,amscd,amsbsy,dsfont,hyperref}
\usepackage[all]{xy}
\usepackage{graphicx,pgf,caption,subcaption}
\usepackage{enumerate}

\DeclareMathAlphabet{\mathpzc}{OT1}{pzc}{m}{it}

\newcommand{\Z}{\mathds Z}
\newcommand{\KK}{\mathds K}
\newcommand{\R}{\mathds R}
\newcommand{\C}{\mathds C}
\newcommand{\Hr}{\mathds H}
\newcommand{\Ca}{\mathds{C}\mathrm{a}}
\newcommand{\SO}{\mathsf{SO}}
\renewcommand{\O}{\mathsf O}
\newcommand{\SU}{\mathsf{SU}}
\newcommand{\U}{\mathsf{U}}
\newcommand{\Sp}{\mathsf{Sp}}
\newcommand{\Spin}{\mathsf{Spin}}
\newcommand{\G}{\mathsf{G}}
\newcommand{\K}{\mathsf{K}}
\renewcommand{\H}{\mathsf{H}}
\newcommand{\Sym}{\operatorname{S}(\wedge^2 V)}

\newcommand{\g}{\mathrm g}
\newcommand{\tr}{\operatorname{tr}}
\newcommand{\Gr}{\operatorname{Gr}_2}
\newcommand{\mm}{\mathfrak m}
\newcommand{\h}{\mathfrak h}
\renewcommand{\k}{\mathfrak k}
\renewcommand{\Im}{\operatorname{Im}}
\renewcommand{\Re}{\operatorname{Re}}
\newcommand{\diag}{\operatorname{diag}}
\newcommand*{\longhookrightarrow}{\ensuremath{\lhook\joinrel\relbar\joinrel\rightarrow}}
\renewcommand{\Sym}{\operatorname{Sym}}
\newcommand{\prs}{\mathpzc{p}_r(\vec s\hspace{1pt} )}

\newcommand{\Ad}{\operatorname{Ad}}

\renewcommand{\1}{\mathds 1}
\renewcommand{\i}{{\rm I}}
\renewcommand{\j}{{\rm J}}
\renewcommand{\k}{{\rm K}}
\renewcommand{\l}{{\rm L}}
\newcommand{\m}{{\rm M}}
\newcommand{\n}{{\rm N}}
\renewcommand{\o}{{\rm O}}

\allowdisplaybreaks

\hypersetup{
    pdftoolbar=true,
    pdfmenubar=true,
    pdffitwindow=false,
    pdfstartview={FitH},
    pdftitle={Flag manifolds with strongly positive curvature},
    pdfauthor={Renato G. Bettiol and Ricardo A. E. Mendes},
    pdfsubject={AMS codes: 53B20, 53C20, 53C21, 53C30, 53C35, 14M15},
    pdfkeywords={}
    pdfnewwindow=true,
    colorlinks=true, 
    linkcolor=blue,
    citecolor=blue,
    urlcolor=black,
}

\newtheorem{theorem}{Theorem}[]

\newtheorem{proposition}[theorem]{Proposition}

\newtheorem{mainthm}{\sc Theorem}

\newtheorem*{convention}{Convention}

\theoremstyle{definition}

\theoremstyle{remark}
\newtheorem{remark}[theorem]{Remark}

\title[Flag manifolds with strongly positive curvature]{Flag manifolds with strongly positive curvature}
\author{Renato G. Bettiol}
\author{Ricardo A. E. Mendes}

\address{\begin{tabular}{lll}
University of Notre Dame & & Universit\"at M\"unster \\
Department of Mathematics & & Mathematisches Institut \\
255 Hurley Building & & Einsteinstr.\ 62 \\
Notre Dame, IN, 46556-4618, USA & & D-48149 M\"unster, Germany\\
\emph{E-mail address}: {\tt rbettiol@nd.edu} & & \emph{E-mail address}: {\tt mendes@uni-muenster.de}
\end{tabular}
}

\allowdisplaybreaks
\numberwithin{equation}{section}
\numberwithin{theorem}{section}

\thanks{The first named author is partially supported by the NSF grant DMS-1209387, USA. The second named author is supported by SFB878 Groups, Geometry \& Actions.}

\date{\today}

\begin{document}
\begin{abstract}
We obtain a complete description of the moduli spaces of homogeneous metrics with strongly positive curvature on the Wallach flag manifolds $W^6$, $W^{12}$ and $W^{24}$, which are respectively the manifolds of complete flags in $\C^3$, $\Hr^3$ and $\Ca^3$. Together with our earlier work~\cite{strongpos}, this concludes the classification of simply-connected homogeneous spaces that admit a homogeneous metric with strongly positive curvature.
\end{abstract}

\maketitle

\section{Introduction}

The study of smooth manifolds with positive sectional curvature ($\sec>0$) is a classical subject in Riemannian Geometry; however, surprisingly little is understood about such manifolds. On the one hand, there are very few known topological obstructions. On the other hand, there is also a dramatic scarcity of examples besides spheres and projective spaces (see \cite{bible} for a survey).
A natural approach to deal with the latter issue is to search for examples with many symmetries~\cite{grove-survey}. Perhaps the most compelling motivations for this approach are the achievements of the classification of compact simply-connected homogeneous spaces with $\sec>0$, due to Wallach~\cite{wa} in even dimensions, and B\'erard-Bergery~\cite{BB76} in odd dimensions,
see also Wilking and Ziller~\cite{wilking-ziller-hom}. In the process of this classification, Wallach~\cite{wa} discovered three new examples of closed manifolds with $\sec>0$, that are commonly referred to as the \emph{Wallach flag manifolds}:
\begin{equation*}
W^6=\SU(3)/\mathsf T^2, \quad W^{12}=\Sp(3)/\Sp(1)\Sp(1)\Sp(1), \;\; \text{ and }\;\; W^{24}=\mathsf F_4/\Spin(8).
\end{equation*}
These are the manifolds of complete flags in $\C^3$, $\Hr^3$ and $\Ca^3$ respectively, and are the total spaces of homogeneous sphere bundles over the corresponding projective planes $\C P^2$, $\Hr P^2$ and $\Ca P^2$, see Section~\ref{sec:flags} for details.

In this paper, we revisit the moduli spaces of positively curved homogeneous metrics on the Wallach flag manifolds (obtained by Valiev~\cite{valiev}), under the light of \emph{strongly positive curvature}. A Riemannian manifold $(M,\g)$ has strongly positive curvature if there exists $\omega\in\Omega^4(M)$ such that the modified curvature operator $(R+\omega)\colon\wedge^2 TM^*\to\wedge^2 TM^*$ is positive-definite.
A systematic study of this condition was initiated by the authors in \cite{strongpos}, propelled by the discovery that it is preserved under Riemannian submersions and Cheeger deformations. Its origins are reminiscent of the work of Thorpe~\cite{ThorpeJDG,Thorpe72}, and such notion was also used by authors including P\"uttmann~\cite{puttmann}, and Grove, Verdiani and Ziller~\cite{p2}, who coined the term.
The interest in this curvature condition is largely due to it being intermediate between $\sec>0$ and positive-definiteness of the curvature operator. We recall that, while many questions on manifolds with $\sec>0$ remain open, the only simply-connected manifolds to admit positive-definite curvature operator are spheres~\cite{BW}. Strongly positive curvature has the potential to improve the understanding of the gap between these classes. 

Our main result is a precise description of the moduli spaces of homogeneous metrics with strongly positive curvature on the Wallach flag manifolds. For the sake of brevity, we state it in terms of the moduli spaces of homogeneous metrics with $\sec>0$, originally obtained by Valiev~\cite{valiev}, see also P\"uttmann~\cite{puttmann,puttmann2}; however, we stress that our proof is independent of these previous results.

\begin{mainthm}\label{mainthm}
A homogenous metric on $W^6$ or $W^{12}$ has strongly positive curvature if and only if it has $\sec>0$. A homogenous metric on $W^{24}$ has strongly positive curvature if and only if it has $\sec>0$ and does not submerge onto $\Ca P^2$.
\end{mainthm}

In the above statement, a homogeneous metric $\g$ on $W^{24}$ is said to \emph{submerge} onto $\Ca P^2$ if at least one of the $3$ natural bundle projections $(W^{24},\g)\to\Ca P^2$ is a Riemannian submersion, see Subsection~\ref{ssec:hommetrics}.


The main motivation for Theorem \ref{mainthm}, besides its intrinsic interest, is that it completes the classification of closed simply-connected homogeneous spaces that admit a homogeneous metric with strongly positive curvature.
Apart from $W^{24}$, which was left undecided, we proved in our previous work \cite[Thm.\ C]{strongpos} that the only closed simply-connected homogeneous space with $\sec>0$ that does not admit a homogeneous metric with strongly positive curvature is the Cayley plane $\Ca P^2$. 
By verifying that $W^{24}$ admits homogeneous metrics with strongly positive curvature, Theorem~\ref{mainthm} yields 
the following classification result:

\begin{mainthm}\label{classification}
All simply-connected homogeneous spaces with $\sec>0$ admit a homogeneous metric with strongly positive curvature, except for $\Ca P^2$.
\end{mainthm}

Despite not admitting a homogeneous metric with strongly positive curvature, $\Ca P^2$ may carry \emph{nonhomogeneous} metrics with this property. It is worth stressing that it is not currently known if all Riemannian manifolds with $\sec>0$ admit a metric with strongly positive curvature, see \cite[Sec.\ 6]{strongpos}. In fact, the only currently known obstructions to strongly positive curvature are obstructions to $\sec>0$.

Let us give a brief outline of the strategy for the proof of Theorem~\ref{mainthm}. The starting point is to explicitly compute the modified curvature operators $(R+\omega)\colon\wedge^2 TW^\bullet\to\wedge^2 TW^\bullet$ of homogeneous metrics on the Wallach flag manifolds $W^\bullet$. Up to averaging using the isometric group action, $\omega$ can be assumed invariant. Thus, $R+\omega$ is an equivariant symmetric linear operator, and hence, by Schur's Lemma, has a certain block diagonal canonical form based on the decomposition of $\wedge^2 TW^\bullet$ into irreducible factors. In particular, $R+\omega$ can be tested for positive-definiteness using its matrix on a suitable basis of \emph{representatives}. The process to find $\omega$ such that $R+\omega$ is positive-definite employs a first-order argument consisting of two steps. The first step is to find $\omega_0$ such that $R+\omega_0$ is positive-semidefinite; and the second step is to find $\omega'$ with positive-definite restriction to the kernel of $R+\omega_0$. In this case, $\omega=\omega_0+\varepsilon\,\omega'$ satisfies the desired property for any sufficiently small $\varepsilon>0$. The only cases in which this process fails are those corresponding to normal homogeneous metrics, and the homogeneous metrics on $W^{24}$ that submerge onto $\Ca P^2$. These do not have strongly positive curvature as a consequence of \cite[Thm.\ A, C]{strongpos}.

As a byproduct of the above proof (specifically, the first step of finding $\omega_0$ in the first-order argument), we obtain a complete description of the moduli spaces of homogeneous metrics with strongly \emph{nonnegative} curvature on these manifolds:

\begin{mainthm}\label{strongnonneg}
A homogenous metric on $W^6$, $W^{12}$ or $W^{24}$ has strongly nonnegative curvature if and only if it has $\sec\geq0$.
\end{mainthm}

This result motivates the question of which constructions of metrics with $\sec\geq0$ yield metrics with strongly nonnegative curvature, which will be addressed in \cite{strongnonneg}.

This paper is organized as follows. In Section~\ref{sec:strongpos}, we briefly recall the definition of strongly positive curvature and provide background on this condition for homogeneous spaces. Section~\ref{sec:flags} contains information on the Wallach flag manifolds and their algebraic structure. Modified curvature operators of homogeneous metrics on these manifolds are computed in Section~\ref{sec:curvops}, in terms of complete lists of representative vectors. Finally, in Section~\ref{sec:moduli}, we determine the above moduli spaces (see Propositions~\ref{prop:modulisec}, \ref{prop:strnneg} and \ref{prop:strpos}); which, in particular, prove Theorems~\ref{mainthm}, \ref{classification} and \ref{strongnonneg}.

\medskip
\noindent
{\bf Acknowledgements.}
It is a pleasure to thank Karsten Grove and Wolfgang Ziller for valuable suggestions. We also acknowledge the hospitality of the Mathematisches Forschungsinstitut Oberwolfach, where many results in this paper were proved.

\section{Strongly positive curvature}\label{sec:strongpos}

In this section, we briefly discuss the notion of modified curvature operators, leading to the definition of \emph{strongly positive curvature}; for details, see \cite{strongpos}. Furthermore, we recall the algebraic expression for curvature operators of homogeneous spaces, according to the formulation of P\"uttmann~\cite{puttmann}.

\subsection{Modified curvature operators}
Let $(M,\g)$ be a Riemannian manifold, and use the inner products induced by $\g$ to identify each $\wedge^k T_pM$ with its dual $\wedge^k T_pM^*$ for every $k$.
Denote by $\Sym^2(\wedge^2 T_pM)$ the space of symmetric linear operators $S\colon\wedge^2T_pM\to \wedge^2 T_pM$, and let $\mathfrak b\colon\Sym^2(\wedge^2 T_pM)\to\wedge^4 T_pM$ be the \emph{Bianchi map}, which associates to each $S\in\Sym^2(\wedge^2 T_pM)$ the $4$-form $\mathfrak b(S)\in\wedge^4 T_pM$ given by
\begin{equation*}
\mathfrak b(S)(X,Y,Z,W):=\tfrac13\Big(\langle S(X\wedge Y),Z\wedge W\rangle + \langle S(Y\wedge Z),X\wedge W\rangle + \langle S(Z\wedge X),Y\wedge W\rangle\Big).
\end{equation*}
The space $\wedge^4 T_pM$ is identified with a subspace of $\Sym^2(\wedge^2 T_pM)$, by
associating each $4$-form $\omega\in\wedge^4 T_pM$ with the operator $\omega\colon\wedge^2T_pM\to \wedge^2 T_pM$ given by
\begin{equation*}
\langle \omega(X\wedge Y),Z\wedge W\rangle:=\omega(X,Y,Z,W).
\end{equation*}
The orthogonal projection onto this subspace is given by the Bianchi map, which hence determines an orthogonal decomposition $\Sym^2(\wedge^2 T_pM)=\ker\mathfrak b\oplus\wedge^4 T_pM$.

The curvature operator $R\colon\wedge^2T_pM\to \wedge^2 T_pM$ is an element of $\ker\mathfrak b$, and the operators $(R+\omega)\in\Sym^2(\wedge^2 T_pM)$, that project onto $R$ under $\mathfrak b$, are called \emph{modified curvature operators}, following P\"uttmann~\cite{puttmann}. The crucial feature of modified curvature operators $R+\omega$ is that they induce the \emph{same} sectional curvature function
\begin{equation*}
\sec\colon\Gr(T_pM)\to\R, \quad \sec(X\wedge Y)=\langle R(X\wedge Y),X\wedge Y\rangle,
\end{equation*}
as the original curvature operator, since $\omega$ vanishes when restricted to $\Gr(T_pM)$.
Here, $\Gr(T_pM):=\{X\wedge Y\in\wedge^2 T_pM:\|X\wedge Y\|^2=1\}$ denotes the Grassmannian of (oriented) $2$-planes on $T_pM$.
In particular, if there exists $\omega\in\wedge^4 T_pM$ such that $(R+\omega)\in\Sym^2(\wedge^2 T_pM)$ is positive-definite, then the sectional curvature of any plane in $T_pM$ is positive. This is sometimes referred to as \emph{Thorpe's trick}.

The manifold $(M,\g)$ is said to have \emph{strongly positive curvature} (respectively, \emph{strongly nonnegative curvature}) if, for all $p\in M$, there exists a $4$-form $\omega\in\wedge^4 T_pM$ such that the modified curvature operator $(R+\omega)\colon\wedge^2 T_pM\to\wedge^2 T_pM$ is positive-definite (respectively, positive-semidefinite). By the above, this condition is clearly intermediate between positive-definiteness (respectively, positive-semidefiniteness) of the curvature operator and $\sec>0$ (respectively, $\sec\geq0$).

\subsection{Homogeneous spaces}\label{subsec:hommetrics}
Let $\G$ be a compact Lie group with bi-invariant metric $Q$, and $\H$ a closed subgroup. Denoting the corresponding Lie algebras by $\mathfrak g$ and $\mathfrak h$, let $\mathfrak m$ be the subspace such that $\mathfrak g=\mathfrak h\oplus\mathfrak m$ is a $Q$-orthogonal direct sum.
Recall that the tangent space to the homogeneous space $\G/\H$ at the identity class $[e]\in\G/\H$ is identified with $\mathfrak m$, and the isotropy representation of $\H$ on $T_{[e]}(\G/\H)\cong\mathfrak m$ corresponds to the adjoint representation $\Ad\colon\H\to\SO(\mm)$.
In particular, $\G$-invariant metrics on $\G/\H$ are in $1$-to-$1$ correspondence with $\Ad(\H)$-invariant inner products on $\mathfrak m$. Any such inner product $\langle\cdot,\cdot\rangle$ is determined by the $Q$-symmetric $\H$-equivariant automorphism $P\colon\mm\to\mm$ such that $\langle X,Y\rangle=Q(PX,Y)$. 
A formula for the curvature operator of the corresponding $\G$-invariant metric on $\G/\H$ 
can be found in P\"uttmann~\cite[Lemma 3.6]{puttmann},
in terms of the bilinear forms $B_\pm$ given by
\begin{equation*}
B_\pm(X,Y):=\tfrac12\big([X,PY]\mp [PX,Y]\big).
\end{equation*}
By introducing the positive-semidefinite operator $\beta\in\Sym^2(\wedge^2\mathfrak g)$, 
given by
\begin{equation*}
\langle\beta(X\wedge Y), Z\wedge W\rangle:=\tfrac14\langle [X,Y]_\mm,[Z,W]_\mm\rangle,
\end{equation*}
where $V_\mm$ denotes the component in $\mm$ of $V\in\mathfrak g$, and rearranging the above mentioned formula in terms of the Bianchi map, one obtains the following expression:
\begin{equation}\label{eq:curvop}
\begin{aligned}
\langle R(X\wedge Y),Z\wedge W\rangle&=\tfrac12\Big(Q(B_-(X,Y),[Z,W])+Q([X,Y],B_-(Z,W))\Big)\\
 &\quad+Q(B_+(X,W),P^{-1}B_+(Y,Z))\\
 &\quad-Q(B_+(X,Z),P^{-1}B_+(Y,W)) \\
&\quad-3\langle\beta(X\wedge Y),Z\wedge W\rangle+3\mathfrak b(\beta)(X,Y,Z,W).
\end{aligned}
\end{equation}
The above procedure is analogous to the rearrangement of the Gray-O'Neill formula for curvature operators of Riemannian submersions in the proof of \cite[Thm.\ 2.4]{strongpos}.\footnote{%
Compare with the \emph{normal homogeneous case} discussed in \cite[Ex.\ 2.5]{strongpos}, in which $P=\operatorname{Id}$, $B_+=0$ and $B_-(\cdot,\cdot)=[\cdot,\cdot]$. The curvature operator of $(\G/\H,Q|_\mm)$ is given by $R_{\G/\H}=R_{\G}+3\alpha-3\mathfrak b(\alpha)$, see \cite[(2.9)]{strongpos}, where $\alpha=A^*A$ is obtained from the $A$-tensor of the submersion $(\G,Q)\to (\G/\H,Q|_\mm)$. Since $R_{\G}=\alpha+\beta$, it follows that $R_{\G/\H}=4R_{\G}-3\beta+3\mathfrak b(\beta)$, which is the expression given by \eqref{eq:curvop}.
}

\begin{remark}\label{rem:symmetries}
In order to verify whether a homogeneous metric on $\G/\H$ has strongly positive (or nonnegative) curvature, it suffices to consider modified curvature operators $R+\omega$ with $4$-forms $\omega\in\Omega^4(\G/\H)$ that are $\G$-invariant. More precisely, since $R\in\Sym^2(\wedge^2\mathfrak m)$ is $\Ad(\H)$-equivariant and the set of positive-definite operators is convex, a routine averaging argument implies that there exists $\omega\in\wedge^4\mm$ such that $R+\omega$ is positive-definite if and only if there exists an $\Ad(\H)$-equivariant $\overline\omega\in\wedge^4\mm$ such that $R+\overline\omega$ is positive-definite, see \cite[Prop. 2.7]{strongpos} or \cite[Lemma 3.5]{puttmann}.
\end{remark}

\section{Wallach flag manifolds}\label{sec:flags}

In this section, we recall the basic structure of the Wallach flag manifolds and their homogeneous metrics.
Let $\KK$ be one of the real normed division algebras: $\R$ of real numbers, $\C$ of complex numbers, $\Hr$ of quaternions, or $\Ca$ of Cayley numbers (or octonions). A \emph{complete flag} $F$ in $\KK^3$ is a sequence of linear subspaces
\begin{equation*}
F=\big\{\{0\}\subset F^1\subset F^2\subset \KK^3\big\},
\end{equation*}
where $\dim_\KK F^i=i$, and the sets of complete flags in $\KK^3$ have a natural smooth structure. We denote these manifolds by $W^3$, $W^6$, $W^{12}$ and $W^{24}$,
according to the cases $\R$, $\C$, $\Hr$ and $\Ca$ respectively.

\subsection{Homogeneous sphere bundles}
The natural projection map 
\begin{equation*}
\pi\colon W^{3\dim\KK}\longrightarrow\KK P^2, \quad \pi(F)=F^1,
\end{equation*}
is a submersion, 
whose fiber $\pi^{-1}(F^1)$ over a $\KK$-line $F^1$ consists of all $\KK$-planes $F^2$ that contain $F^1$. 
Such $\KK$-planes are in one-to-one correspondence with the $\KK$-lines in the orthogonal complement $(F^1)^\perp$, hence $\pi^{-1}(F^1)$ can be identified with $\KK P^1$.
This makes $W^{3\dim\KK}$ the total space of a sphere bundle over $\KK P^2$; more precisely,
\begin{equation}\label{eq:bundles}
\begin{aligned}
&S^1\longrightarrow W^3 \longrightarrow \R P^2,& \quad &S^2\longrightarrow W^6 \longrightarrow \C P^2,\\
&S^4\longrightarrow W^{12}\longrightarrow \Hr P^2,& \quad &S^8\longrightarrow W^{24}\longrightarrow \Ca P^2.
\end{aligned}
\end{equation}
The isometry group $\G$ of each projective plane $\KK P^2$ above acts transitively on $W^{3\dim\KK}$. Furthermore, the induced subaction of the isotropy subgroup $\K$ of a $\KK$-line $F^1$ is transitive on the corresponding fiber $\pi^{-1}(F^1)$. Altogether, each sphere bundle \eqref{eq:bundles} with $W^{3\dim\KK}$ as total space is a homogeneous bundle of the form
\begin{equation}\label{eq:hombundle}
\K/\H\longrightarrow\G/\H\longrightarrow\G/\K,
\end{equation}
for a triple of compact Lie groups $\H\subset\K\subset\G$, as described in the following table:
\begin{equation}\label{eq:groups}
\arraycolsep=10pt\def\arraystretch{1.2}
\begin{array}{c|c|ccc}
\G/\H & \KK & \H & \K & \G\\ \hline
W^3 & \R & \Z_2\oplus\Z_2 & \O(2) & \SO(3)\\
W^6 & \C & \mathsf T^2 & \U(2) & \SU(3)\\
W^{12} & \Hr & \Sp(1)\Sp(1)\Sp(1) & \Sp(2)\Sp(1) & \Sp(3)\\
W^{24} & \Ca & \Spin(8) & \Spin(9) & \mathsf F_4
\end{array}
\end{equation}
The inclusions $\H\subset\K\subset\G$ for the quaternionic flag $W^{12}$ are given by the obvious matrix block embeddings, and similarly for the real and complex flags $W^3$ and $W^6$:
\begin{equation*}
\arraycolsep=3pt\def\arraystretch{1.3}
\begin{array}{rccccc}
\Z_2\oplus\Z_2\cong\mathsf S(\O(1)\O(1)\O(1))&\subset& \O(2)\cong\mathsf S(\O(2)\O(1)) &\subset &\SO(3)& \\
\mathsf T^2\cong\mathsf S(\U(1)\U(1)\U(1))\,& \subset &\U(2)\cong\mathsf S(\U(2)\U(1)) & \subset &\SU(3).&
\end{array}
\end{equation*}
Regarding the octonionic flag $W^{24}$, we describe the inclusions of $\Spin(8)$ into $\Spin(9)$ into $\mathsf F_4$ following the exceptional survey of Baez~\cite{baez}. The Lie group $\mathsf F_4$ can be realized as the automorphism group of the exceptional Jordan algebra
\begin{equation*}
\h_3(\Ca)=\big\{H\in\operatorname{Mat}_{3\times 3}(\Ca):H^*=H\big\},
\end{equation*}
of Hermitian $3\times 3$ octonionic matrices. As described in \cite[Sec. 3.4]{baez}, this algebra can also be constructed using the scalar, vector and spinor representations of $\Spin(9)$, which is hence a subgroup of $\mathsf F_4$ (see also Harvey~\cite[Thm.\ 14.99]{harvey-book}). The Lie algebra $\mathfrak f_4$ is the algebra of derivations of $\h_3(\Ca)$; in particular, $\mathfrak f_4$ contains a copy of the algebra $\mathfrak g_2$ of derivations of $\Ca$, which is a subalgebra of $\mathfrak{so}(7)\cong\mathfrak{so}(\Im(\Ca))$. More precisely, there is a splitting
\begin{equation}\label{eq:f4}
\mathfrak f_4=\mathfrak{sa}_3(\Ca)\oplus\mathfrak g_2,
\end{equation}
where $\mathfrak{sa}_3(\Ca):=\big\{A\in\operatorname{Mat}_{3\times 3}(\Ca):A^*=-A, \tr(A)=0\big\}$. For more details on the above, including formulas for the Lie bracket on $\mathfrak f_4$, see Baez~\cite[Sec. 4.2]{baez}.

\subsection{Homogeneous metrics}
\label{ssec:hommetrics}
As the above groups $\G$ are simple, by Schur's Lemma, they admit a unique bi-invariant metric $Q$, up to rescaling. We henceforth fix:
\begin{equation}\label{eq:q}
Q(X,Y)=\begin{cases}
\tfrac12\Re(\tr (XY^*)) & \text{ if } \mathsf G=\SO(3), \SU(3), \Sp(3),\\
\tfrac12\Re(\tr (X_1Y_1^*))+\tfrac18\tr (X_2Y_2^*) & \text{ if } \mathsf G=\mathsf F_4,
\end{cases}
\end{equation}
where $X_1\in\mathfrak{sa}_3(\Ca)$ and $X_2\in\mathfrak g_2\subset\mathfrak{so}(7)$ denote the components of $X\in\mathfrak f_4$, according to \eqref{eq:f4}.

In all cases, the $\Ad(\H)$-representation $\mm$ has $3$ irreducible factors
$\mm=V_1\oplus V_2\oplus V_3,$
each isomorphic to $\KK$ as a real vector space.
These irreducible factors can be parametrized by the skew-Hermitian matrices with zero diagonal entries
\begin{equation*}
\KK\oplus\KK\oplus\KK\ni (x_1,x_2,x_3)\longmapsto
\begin{pmatrix}
0 & x_3 & -\overline{x_2} \\
-\overline{x_3} & 0 & x_1\\
x_2 & -\overline{x_1} & 0
\end{pmatrix}\in V_1\oplus V_2\oplus V_3.
\end{equation*}
Consider the standard basis of $\KK$ given by
\begin{equation*}
\arraycolsep=5pt\def\arraystretch{1.3}
\begin{array}{ll}
\{1\} &\text{ for }\KK=\R, \\
\{1,i\} &\text{ for }\KK=\C, \\
\{1,i,j,k\} &\text{ for }\KK=\Hr, \\
\{1,i,j,k,l,m,n,o\} &\text{ for }\KK=\Ca.
\end{array}
\end{equation*}
From these, we construct a $Q$-orthonormal basis of $\mm=V_1\oplus V_2\oplus V_3\cong\KK\oplus\KK\oplus\KK$. To denote the elements of this basis, we use the symbols
\begin{equation*}
\1_r,\i_r,\j_r,\k_r,\l_r,\m_r,\n_r,\o_r,\qquad1\leq r\leq 3.
\end{equation*}
For example, in the case $\KK=\C$, the space $\mm=V_1\oplus V_2\oplus V_3$ is spanned by:
\begin{equation*}
\begin{array}{lll}
\1_1=\begin{pmatrix}
0 & 0 & 0 \\
0 & 0 & 1\\
0 & -1 & 0
\end{pmatrix},
& \1_2=\begin{pmatrix}
0 & 0 & -1 \\
0 & 0 & 0\\
1 & 0 & 0
\end{pmatrix},
&\1_3=\begin{pmatrix}
0 & 1 & 0 \\
-1 & 0 & 0\\
0 & 0 & 0
\end{pmatrix}, \\
\rule{0pt}{6.5ex}
\i_1=\begin{pmatrix}
0 & 0 & 0 \\
0 & 0 & i\\
0 & i & 0
\end{pmatrix},
&\i_2=\begin{pmatrix}
0 & 0 & i \\
0 & 0 & 0\\
i & 0 & 0
\end{pmatrix},
&\i_3=\begin{pmatrix}
0 & i & 0 \\
i & 0 & 0\\
0 & 0 & 0
\end{pmatrix}.
\end{array}
\end{equation*}

Since $V_r$, $1\leq r\leq3$, are irreducible and nonisomorphic, Schur's Lemma implies that $\G$-invariant metrics on $\G/\H$ are parametrized by $3$ positive numbers. We denote these by $\vec s=(s_1,s_2,s_3)$, 
so that the corresponding $\G$-invariant metric is given by
\begin{equation*}
\g_{\vec s}:=s_1^2\, Q|_{V_1}+s_2^2\, Q|_{V_2}+s_3^2\, Q|_{V_3}, \quad \text{at } T_{[e]}(\G/\H)=V_1\oplus V_2\oplus V_3.
\end{equation*}

Note that the Lie algebra of the intermediate subgroup $\K$, according to the above conventions, is given by $\mathfrak k=\mathfrak h\oplus  V_3$. As $\g_{\vec{s}}$ is $\Ad(\K)$-invariant if and only $s_1=s_2$, the homogeneous bundle \eqref{eq:hombundle} is a Riemannian submersion if and only if $s_1=s_2$. More generally, 
for each $1\leq r\leq 3$, there exists a copy of $\K$ in $\G$ whose Lie algebra is $\mathfrak{h}\oplus V_r$. Each such choice for $\K$ determines a bundle \eqref{eq:bundles}, which is a Riemannian submersion onto the corresponding projective plane (with its standard Fubini metric) if and only if $s_{r+1}=s_{r+2}$, where indices are understood modulo $3$.

Finally, for any fixed $\vec s$, observe that the image of each natural inclusion
\begin{equation}
\label{eq:totgeod}
(W^3,\g_{\vec s})\longhookrightarrow (W^6,\g_{\vec s})\longhookrightarrow (W^{12},\g_{\vec s})\longhookrightarrow (W^{24},\g_{\vec s})
\end{equation}
is the fixed point set of an isometry, and hence a \emph{totally geodesic} submanifold.

\section{Modified curvature operators}
\label{sec:curvops}

The goal of this section is to derive explicit formulas for the modified curvature operators $R+\omega$ of all homogeneous metrics $\g_{\vec s}$ on the Wallach flag manifolds, where $\omega$ is a general $\Ad(\H)$-invariant $4$-form (see Remark~\ref{rem:symmetries}).

Let us briefly outline an important computational simplification that will be used throughout this section.
Since $(R+\omega)\colon\wedge^2\mm\to\wedge^2\mm$ is $\Ad(\H)$-equivariant, Schur's Lemma implies that it admits a
block diagonal decomposition, with blocks corresponding to the isotypic components of $\wedge^2\mm$. Each isotypic component is the sum of $n$ copies of an irreducible $\H$-representation $V$ of dimension $d$.
As $R+\omega$ is symmetric, 
its restriction to this isotypic component 
is, in a suitable basis, the tensor product of the $d\times d$ identity matrix with a certain $n\times n$ symmetric matrix $A$.  
In particular, the corresponding block of $R+\omega$ is positive-definite (respectively, positive-semidefinite) if and only $A$ is positive-definite (respectively, positive-semidefinite).

In order to compute each of the symmetric matrices $A$, we choose a \emph{representative} vector $v_1\neq 0$ in one copy of $V$, and then produce additional representatives $v_2, v_3, \ldots$ in the remaining copies of $V$ by taking images of $v_1$ under appropriate isomorphisms of representations. Proceeding in this way through all isotypic components, we obtain a complete list of representatives $\{v_i\}$, such that the restriction of $R+\omega$ to the space of representatives $\operatorname{span}\{v_i\}$ is positive-definite (respectively, positive-semidefinite) if and only if $R+\omega$ is positive-definite (respectively, positive-semidefinite).
We denote by $\widehat{R}_{W^\bullet}(\vec s,\omega)$ this restriction to the space of representatives of the curvature operator $R$ of the flag manifold $(W^\bullet,\g_{\vec s})$ modified by the invariant $4$-form $\omega$.

\begin{convention}
In what follows, subindices are always taken \emph{modulo} $3$, and
\begin{equation}\label{eq:s}
s:=2(s_1s_2+s_1s_3+s_2s_3)-(s_1^2+s_2^2+s_3^2).
\end{equation}
\end{convention}

We skip this computation for the real flag manifold $W^3$, since strongly positive curvature and $\sec>0$ are equivalent in dimensions $\leq4$, see \cite[Prop.\ 2.2]{strongpos} or \cite{Thorpe72}.

\subsection{Complex flag manifold \texorpdfstring{$W^6$}{W6}}
The isotropy group $\H$ is the maximal torus $\mathsf T^2$ of $\G=\SU(3)$, see \eqref{eq:groups}. A computation with weights shows that the $\H$-representation $\wedge^2\mm$ decomposes as the sum of $3$ copies of the trivial representation and single copies of $6$ mutually nonisomorphic $2$-dimensional representations. The following $9$ representative vectors form a complete list, as described above:
\begin{align*}
&\1_r\wedge\i_r, & 1\leq r\leq 3,\\
-&\tfrac{1}{\sqrt2}\big(\1_{r+1}\wedge\1_{r+2}\big)+\tfrac{1}{\sqrt2}\big(\i_{r+1}\wedge\i_{r+2}\big), & 1\leq r\leq 3,\\
&\tfrac{1}{\sqrt2}\big(\1_{r+1}\wedge\1_{r+2}\big)+\tfrac{1}{\sqrt2}\big(\i_{r+1}\wedge\i_{r+2}\big), & 1\leq r\leq 3.
\intertext{Furthermore, the following determine a basis of the $\H$-invariant elements of $\wedge^4\mm$:}
& \xi_r:=\1_{r+1}\wedge\i_{r+1}\wedge\1_{r+2}\wedge\i_{r+2}, &1\leq r\leq 3,
\end{align*}
and we denote a general invariant $4$-form by:
\begin{equation}\label{eq:4formw6}
\xi=a_1\,\xi_1+a_2\,\xi_2+a_3\,\xi_3.
\end{equation}

The restriction $\widehat R$ of the modified curvature operator $R+\xi$ of the homogeneous manifold $(W^6,\g_{\vec s})$ to the subspace of $\wedge^2\mm$ spanned by the above representative vectors can be computed using \eqref{eq:curvop}, and results in the block diagonal matrix
\begin{equation*}
\widehat{R}_{W^6}=\diag\left(\widehat{R}_{W^6}^1, \widehat{R}_{W^6}^2, \widehat{R}_{W^6}^3\right)\!,
\end{equation*}
where the blocks, listed in the same order as the representatives, are given by:
\begin{align}
&\widehat{R}_{W^6}^1:=\begin{pmatrix}
4 s_1 & -\frac{s}{2 s_3}+a_3 & -\frac{s}{2 s_2}+a_2 \\
-\frac{s}{2 s_3}+a_3 & 4 s_2 & -\frac{s}{2 s_1}+a_1 \\
-\frac{s}{2 s_2}+a_2 & -\frac{s}{2 s_1}+a_1 & 4 s_3
\end{pmatrix}\!,\label{eq:R1w6} \\
&\widehat{R}_{W^6}^2:=\diag\big(s_{r+1}+s_{r+2}-s_r+a_r, \; 1\leq r\leq 3\big),\label{eq:R2w6} \\
&\widehat{R}_{W^6}^3:=\diag\left(\tfrac{(s_{r+1} - s_{r+2})^2 - s_r^2}{2 s_r}-a_r, \; 1\leq r\leq 3\right)\!.\label{eq:R3w6}
\end{align}

\subsection{Quaternionic flag manifold \texorpdfstring{$W^{12}$}{W12}}
The isotropy group is $\H=\Sp(1)^3$, see \eqref{eq:groups}, and the $\H$-representation $\wedge^2\mm$ decomposes into a total of $9$ isotypic components. More precisely, there are $2$ copies each of $3$ irreducible nonisomorphic representations of dimension $3$, $3$ distinct irreducible representations of dimension $12$, and $3$ distinct irreducible representations of dimension $4$. The following $12$ representative vectors form a complete list, as described above:
\begin{align*}
\tfrac{1}{\sqrt2}&\big(\1_{r+1}\wedge\i_{r+1}-\j_{r+1}\wedge\k_{r+1}\big),\; \tfrac{1}{\sqrt2}\big(\1_{r+2}\wedge\i_{r+2}+\j_{r+2}\wedge\k_{r+2}\big), &1\leq r\leq 3,\\
\tfrac{1}{2}&\big(\1_{r+1}\wedge \1_{r+2}-\i_{r+1}\wedge\i_{r+2}-\j_{r+1}\wedge\j_{r+2}-\k_{r+1}\wedge\k_{r+2}\big), &1\leq r\leq 3,\\
\tfrac{1}{\sqrt2}&\big(\1_{r+1}\wedge \1_{r+2}+\k_{r+1}\wedge \k_{r+2}\big), & 1\leq r\leq 3.
\intertext{Furthermore, the following determine a basis of the $\H$-invariant elements of $\wedge^4\mm$:}
\phi_r:= \; &(\1_{r+1}\wedge \i_{r+1}-\j_{r+1}\wedge \k_{r+1})\wedge(\1_{r+2}\wedge\i_{r+2}+\j_{r+2}\wedge\k_{r+2})&\\
\quad +&(\1_{r+1}\wedge \j_{r+1}+\i_{r+1}\wedge\k_{r+1})\wedge(\1_{r+2}\wedge\j_{r+2}-\i_{r+2}\wedge\k_{r+2})&\\
\quad +&(\1_{r+1}\wedge\k_{r+1}-\i_{r+1}\wedge\j_{r+1})\wedge(\1_{r+2}\wedge\k_{r+2}+\i_{r+2}\wedge\j_{r+2}),& 1\leq r\leq 3,\\
\psi_r :=\; &\1_r\wedge\i_r\wedge\j_r\wedge\k_r,& 1\leq r\leq 3,
\end{align*}
and we denote a general invariant $4$-form by $\phi+\psi$, where:
\begin{equation}\label{eq:4formw12}
\begin{aligned}
\phi&=a_1\,\phi_1+a_2\,\phi_2+a_3\,\phi_3,\\
\psi&=b_1\,\psi_1+b_2\,\psi_2+b_3\,\psi_3.
\end{aligned}
\end{equation}

The restriction $\widehat R$ of the modified curvature operator $R+\phi+\psi$ of the homogeneous manifold $(W^{12},\g_{\vec s})$ to the subspace of $\wedge^2\mm$ spanned by the above representative vectors can be computed using \eqref{eq:curvop}, and results in the block diagonal matrix
\begin{equation*}
\widehat{R}_{W^{12}}=\diag\left(\widehat{R}_{W^{12}}^1, \widehat{R}_{W^{12}}^2, \widehat{R}_{W^{12}}^3\right)\!,
\end{equation*}
where the blocks, listed in the same order as the representatives, are given by:
\begin{align}
&\widehat{R}_{W^{12}}^1:=
\diag\left(
\begin{pmatrix}
4s_{r+1}-b_{r+1} & -\frac{s}{s_r}+2a_r\\
-\frac{s}{s_r}+2a_r& 4s_{r+2}+b_{r+2}
\end{pmatrix}, \; 1\leq r\leq 3\right)\!,\label{eq:R1w12} \\
&\widehat{R}_{W^{12}}^2:=\diag\big(s_{r+1}+s_{r+2}-s_r+\tfrac{s}{2s_r}+3a_r, \; 1\leq r\leq 3\big),\label{eq:R2w12} \\
&\widehat{R}_{W^{12}}^3:=\diag\left(\tfrac{(s_{r+1} - s_{r+2})^2 - s_r^2}{2 s_r}-a_r, \; 1\leq r\leq 3\right)\!.\label{eq:R3w12}
\end{align}

\subsection{Octonionic flag manifold \texorpdfstring{$W^{24}$}{W24}}
The isotropy group is $\H=\Spin(8)$, see \eqref{eq:groups}, and the $\H$-representation $\wedge^2\mm$ decomposes as the sum of $3$ copies of an irreducible representation of dimension $28$, $3$ distinct irreducible representations of dimension $56$, and $3$ irreducible representations of dimension $8$. The following $9$ representative vectors form a complete list, as described above:
\begin{align*}
\tfrac{1}{\sqrt2}&\big(\j_{r}\wedge\l_{r}+\k_{r}\wedge\m_{r}\big), &1\leq r\leq 3,\\
\tfrac{1}{2\sqrt2}&\big(\1_{r+1}\wedge \1_{r+2}-\i_{r+1}\wedge\i_{r+2}-\j_{r+1}\wedge\j_{r+2}-\k_{r+1}\wedge\k_{r+2}&\\
&\phantom{\big(}-\l_{r+1}\wedge \l_{r+2}-\m_{r+1}\wedge\m_{r+2}-\n_{r+1}\wedge\n_{r+2}-\o_{r+1}\wedge\o_{r+2}\big), &1\leq r\leq 3,\\
\tfrac{1}{\sqrt2}&\big(\1_{r+1}\wedge \1_{r+2}+\o_{r+1}\wedge \o_{r+2}\big), & 1\leq r\leq 3.
\end{align*}
Furthermore, the following determine a basis of the $\H$-invariant elements of $\wedge^4\mm$:
\begin{align*}
\zeta_r \, :=\, &(\1_{r+1}\wedge \i_{r+1}-\j_{r+1}\wedge \k_{r+1})\wedge(\1_{r+2}\wedge\i_{r+2}+\j_{r+2}\wedge\k_{r+2})\\
+&(\1_{r+1}\wedge \j_{r+1}+\i_{r+1}\wedge\k_{r+1})\wedge(\1_{r+2}\wedge\j_{r+2}-\i_{r+2}\wedge\k_{r+2})\\
+&(\1_{r+1}\wedge\k_{r+1}-\i_{r+1}\wedge\j_{r+1})\wedge(\1_{r+2}\wedge\k_{r+2}+\i_{r+2}\wedge\j_{r+2})\\
+&(\1_{r+1}\wedge\l_{r+1}+\i_{r+1}\wedge\m_{r+1})\wedge(\1_{r+2}\wedge\l_{r+2}-\i_{r+2}\wedge\m_{r+2})\\
+&(\1_{r+1}\wedge\m_{r+1}-\i_{r+1}\wedge\l_{r+1})\wedge(\1_{r+2}\wedge\m_{r+2}+\i_{r+2}\wedge\l_{r+2})\\
+&(\1_{r+1}\wedge\n_{r+1}-\i_{r+1}\wedge\o_{r+1})\wedge(\1_{r+2}\wedge\n_{r+2}+\i_{r+2}\wedge\o_{r+2})\\
+&(\1_{r+1}\wedge\o_{r+1}+\i_{r+1}\wedge\n_{r+1})\wedge(\1_{r+2}\wedge\o_{r+2}-\i_{r+2}\wedge\n_{r+2})\\
\rule{0pt}{2.7ex}
+&(\1_{r+1}\wedge \i_{r+1}+\j_{r+1}\wedge \k_{r+1})\wedge(\l_{r+2}\wedge\m_{r+2}-\n_{r+2}\wedge\o_{r+2})\\
+&(\1_{r+1}\wedge \j_{r+1}-\i_{r+1}\wedge\k_{r+1})\wedge(\l_{r+2}\wedge\n_{r+2}+\m_{r+2}\wedge\o_{r+2})\\
+&(\1_{r+1}\wedge\k_{r+1}+\i_{r+1}\wedge\j_{r+1})\wedge(\l_{r+2}\wedge\o_{r+2}-\m_{r+2}\wedge\n_{r+2})\\
-&(\1_{r+1}\wedge\l_{r+1}-\i_{r+1}\wedge\m_{r+1})\wedge(\j_{r+2}\wedge\n_{r+2}+\k_{r+2}\wedge\o_{r+2})\\
-&(\1_{r+1}\wedge\m_{r+1}+\i_{r+1}\wedge\l_{r+1})\wedge(\j_{r+2}\wedge\o_{r+2}-\k_{r+2}\wedge\n_{r+2})\\
+&(\1_{r+1}\wedge\n_{r+1}+\i_{r+1}\wedge\o_{r+1})\wedge(\j_{r+2}\wedge\l_{r+2}-\k_{r+2}\wedge\m_{r+2})\\
+&(\1_{r+1}\wedge\o_{r+1}-\i_{r+1}\wedge\n_{r+1})\wedge(\j_{r+2}\wedge\m_{r+2}+\k_{r+2}\wedge\l_{r+2})\\
\rule{0pt}{2.7ex}
-&(\j_{r+1}\wedge \l_{r+1}-\k_{r+1}\wedge \m_{r+1})\wedge(\1_{r+2}\wedge\n_{r+2}-\i_{r+2}\wedge\o_{r+2})\\
-&(\j_{r+1}\wedge \m_{r+1}+\k_{r+1}\wedge\l_{r+1})\wedge(\1_{r+2}\wedge\o_{r+2}+\i_{r+2}\wedge\n_{r+2})\\
+&(\j_{r+1}\wedge\n_{r+1}+\k_{r+1}\wedge\o_{r+1})\wedge(\1_{r+2}\wedge\l_{r+2}+\i_{r+2}\wedge\m_{r+2})\\
+&(\j_{r+1}\wedge\o_{r+1}-\k_{r+1}\wedge\n_{r+1})\wedge(\1_{r+2}\wedge\m_{r+2}-\i_{r+2}\wedge\l_{r+2})\\
-&(\l_{r+1}\wedge\m_{r+1}-\n_{r+1}\wedge\o_{r+1})\wedge(\1_{r+2}\wedge\i_{r+2}-\j_{r+2}\wedge\k_{r+2})\\
-&(\l_{r+1}\wedge\n_{r+1}+\m_{r+1}\wedge\o_{r+1})\wedge(\1_{r+2}\wedge\j_{r+2}+\i_{r+2}\wedge\k_{r+2})\\
-&(\l_{r+1}\wedge\o_{r+1}-\m_{r+1}\wedge\n_{r+1})\wedge(\1_{r+2}\wedge\k_{r+2}-\i_{r+2}\wedge\j_{r+2})\\
\rule{0pt}{2.7ex}
-&(\j_{r+1}\wedge \l_{r+1}+\k_{r+1}\wedge \m_{r+1})\wedge(\j_{r+2}\wedge\l_{r+2}+\k_{r+2}\wedge\m_{r+2})\\
-&(\j_{r+1}\wedge \m_{r+1}-\k_{r+1}\wedge\l_{r+1})\wedge(\j_{r+2}\wedge\m_{r+2}-\k_{r+2}\wedge\l_{r+2})\\
-&(\j_{r+1}\wedge\n_{r+1}-\k_{r+1}\wedge\o_{r+1})\wedge(\j_{r+2}\wedge\n_{r+2}-\k_{r+2}\wedge\o_{r+2})\\
-&(\j_{r+1}\wedge\o_{r+1}+\k_{r+1}\wedge\n_{r+1})\wedge(\j_{r+2}\wedge\o_{r+2}+\k_{r+2}\wedge\n_{r+2})\\
-&(\l_{r+1}\wedge\m_{r+1}+\n_{r+1}\wedge\o_{r+1})\wedge(\l_{r+2}\wedge\m_{r+2}+\n_{r+2}\wedge\o_{r+2})\\
-&(\l_{r+1}\wedge\n_{r+1}-\m_{r+1}\wedge\o_{r+1})\wedge(\l_{r+2}\wedge\n_{r+2}-\m_{r+2}\wedge\o_{r+2})\\
-&(\l_{r+1}\wedge\o_{r+1}+\m_{r+1}\wedge\n_{r+1})\wedge(\l_{r+2}\wedge\o_{r+2}+\m_{r+2}\wedge\n_{r+2}),\; 1\leq r\leq 3,
\end{align*}
and we denote a general invariant $4$-form by:
\begin{equation}\label{eq:4formw24}
\zeta=a_1\,\zeta_1+a_2\,\zeta_2+a_3\,\zeta_3.
\end{equation}

The restriction $\widehat R$ of the modified curvature operator $R+\zeta$ of the homogeneous manifold $(W^{24},\g_{\vec s})$ to the subspace of $\wedge^2\mm$ spanned by the above representative vectors can be computed using \eqref{eq:curvop}, and results in the block diagonal matrix
\begin{equation*}
\widehat{R}_{W^{24}}=\diag\left(\widehat{R}_{W^{24}}^1, \widehat{R}_{W^{24}}^2, \widehat{R}_{W^{24}}^3\right)\!,
\end{equation*}
where the blocks, listed in the same order as the representatives, are given by:
\begin{align}
&\widehat{R}_{W^{24}}^1:=\begin{pmatrix}
4 s_1 & \frac{s}{s_3}-2a_3 & \frac{s}{s_2}-2a_2 \\
\frac{s}{s_3}-2a_3 & 4 s_2 & \frac{s}{s_1}-2a_1 \\
\frac{s}{s_2}-2a_2 & \frac{s}{s_1}-2a_1 & 4 s_3
\end{pmatrix}\!,\label{eq:R1w24} \\
&\widehat{R}_{W^{24}}^2:=\diag\big(s_{r+1}+s_{r+2}-s_r+\tfrac{3s}{2s_r}+7a_r, \; 1\leq r\leq 3\big), \label{eq:R2w24} \\
&\widehat{R}_{W^{24}}^3:=\diag\left(\tfrac{(s_{r+1} - s_{r+2})^2 - s_r^2}{2 s_r}-a_r, \; 1\leq r\leq 3\right)\!.\label{eq:R3w24}
\end{align}

\begin{remark}
Since there are totally geodesic isometric embeddings \eqref{eq:totgeod} between the Wallach flag manifolds, their modified curvature operators are restrictions of one another. Nevertheless, the above matrices $\widehat R_{W^\bullet}$ are \emph{not} submatrices of one another, as the lists of representative vectors are \emph{not} sublists of one another. In fact, such a choice is not possible due to how the corresponding representations decompose into irreducibles.
We also remark that some invariant $4$-forms of each Wallach flag manifold can be seen as projections of invariant $4$-forms of a larger Wallach flag manifold; namely
 $\zeta_r$ projected to $\wedge^4$ of the subspace tangent to $W^{12}$ equals $\phi_r$, and $\phi_r$ projected to $\wedge^4$ of the subspace tangent to $W^6$ equals $\xi_r$.
\end{remark}

\section{Moduli spaces}
\label{sec:moduli}

In this section, we determine the moduli spaces of homogeneous metrics with strongly nonnegative and strongly positive curvature on the Wallach flag manifolds. 
This is based on a detailed analysis of the matrices $R_{W^\bullet}(\vec s,\omega)$ that encode the action of modified curvature operators on a complete list of representatives, see Section~\ref{sec:curvops}.

We begin by recalling the moduli spaces of homogeneous metrics with $\sec\geq0$ and $\sec>0$, in Proposition~\ref{prop:modulisec}.
We then assign to each $\vec s$ a $4$-form $\omega_0(\vec s)$ with the remarkable property that $R+\omega_0$ is positive-semidefinite if and only if $\g_{\vec s}$ has strongly nonnegative curvature, see \eqref{eq:ar}. This is the key step in determining the moduli spaces of homogeneous metrics with strongly nonnegative curvature (Proposition~\ref{prop:strnneg}), which proves Theorem~\ref{strongnonneg}.
Finally, we combine these results with a first-order perturbation argument to determine the moduli spaces of homogeneous metrics with strongly positive curvature (Proposition~\ref{prop:strpos}), which proves Theorem~\ref{mainthm}.

\begin{convention}
In what follows, we make frequent use of the polynomials
\begin{equation}\label{eq:pr}
\prs:=(s_{r+1}-s_{r+2})^2+2s_r(s_{r+1}+s_{r+2})-3s_r^2.
\end{equation}
\end{convention}

Note that $s=\mathpzc{p}_1(\vec s\hspace{1pt} )+\mathpzc{p}_2(\vec s\hspace{1pt} )+\mathpzc{p}_3(\vec s\hspace{1pt} )$, see \eqref{eq:s}.

\subsection{Sectional curvature}
\label{ssec:sectional}
An important preliminary result to discuss $\sec\geq0$ and $\sec>0$ among homogeneous metrics on the Wallach flag manifolds is to recall the moduli space of left-invariant metrics with these properties on $\SU(2)$. Note that this corresponds to the case of the real flag manifold, as 
$(W^3,\g_{\vec s})$ is locally isometric to $\SU(2)$ equipped with the left-invariant metric for which the canonical orthogonal Killing fields have lengths $s_1$, $s_2$, and $s_3$, see \eqref{eq:groups}.
This metric is known to have $\sec\geq 0$ (respectively, $\sec>0$) if and only if $\prs\geq 0$ (respectively $\prs>0$) for $r=1,2,3$, see \cite[Thm.\ B]{vz}, and Figure~\ref{fig:modulisec} for a graphical representation.

\begin{remark}
The curvature conditions on $\g_{\vec s}$ that we consider are scale-invariant on $\vec s$, and permuting the elements of $\vec s=(s_1,s_2,s_3)$ results in isometric homogeneous metrics. 
In particular, the \emph{moduli space} of homogeneous metrics $\g_{\vec s}$ satisfying a given curvature condition is to be understood up to these ambiguities, which can be removed by imposing gauge fixing conditions, for instance, $s_1\leq s_2\leq s_3$ and $s_1+s_2+s_3=1$.
The latter normalization condition is used in Figure~\ref{fig:modulisec}, to facilitate the graphical representation. Note that the permutation invariance is evident in Figure~\ref{fig:modulisec}, where there are $6$ isometric copies of the moduli space $\{\g_{\vec s}:\sec_{\g_{\vec s}}\geq0\}$.
\end{remark}

\begin{figure}[htf]
\vspace{-0.25cm}
\centering
\includegraphics[scale=0.7]{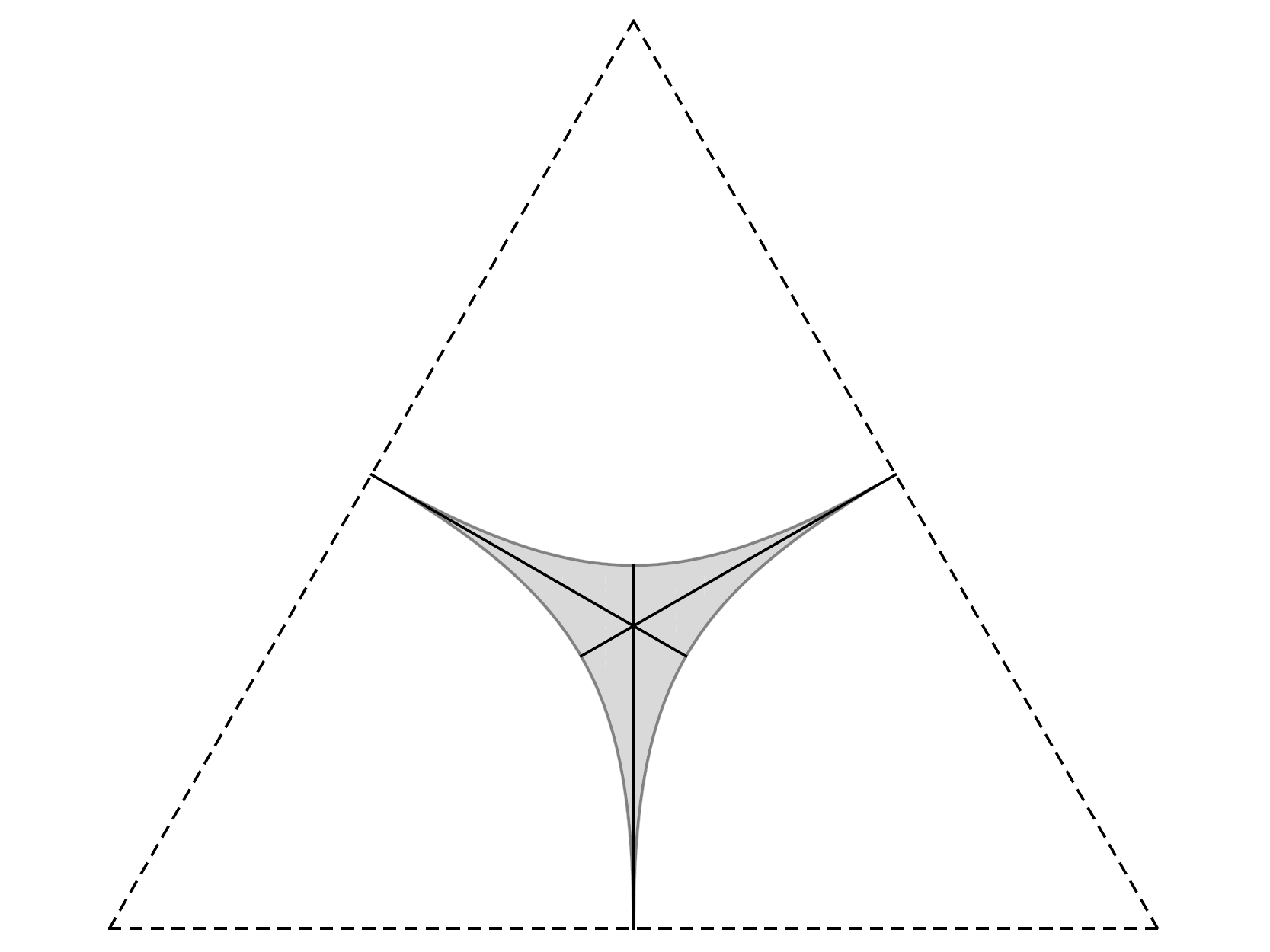}
\caption{
The intersection of $\{\vec s\in\R^3:\prs\geq0\}$ with the plane $s_1+s_2+s_3=1$.
Dotted segments indicate where $s_r=0$, and continuous segments indicate where two of the $s_r$ are equal.
}
\label{fig:modulisec}
\end{figure}

\begin{proposition}\label{prop:modulisec}
The homogeneous metric $\g_{\vec s}$ on $W^6$, $W^{12}$, or $W^{24}$, has $\sec\geq0$ if and only if $\prs\geq0$, $r=1,2,3$. Furthermore, $\g_{\vec s}$ has $\sec>0$ if and only if $\prs>0$, $r=1,2,3$, and $s_r$ are not all equal.
\end{proposition}

\begin{proof}
See Valiev~\cite[Thm.\ 2]{valiev}, and also P\"uttmann~\cite[Thm.\ 3.1]{puttmann2}.
\end{proof}

\begin{remark}
The central point in Figure~\ref{fig:modulisec}, where $s_1=s_2=s_3$, corresponds to a multiple of the bi-invariant metric $Q$ on $\G$, see \eqref{eq:q}. Thus, the induced metric on $W^\bullet$ is a \emph{normal} homogeneous metric, which is well-known to have $\sec\geq0$ but not $\sec>0$. Note also that the moduli space of homogeneous metrics with $\sec\geq0$ on $W^\bullet$ is star-shaped around the normal homogeneous metric, cf.\ \cite[Prop.\ 1.1]{schwa-tapp}.
\end{remark}

\begin{remark}
The fact that the above moduli spaces for $W^6$, $W^{12}$ and $W^{24}$ coincide can be explained using a generalization of the embeddings \eqref{eq:totgeod}, by showing that planes with extremal sectional curvature can be moved via isometries to become tangent to certain totally geodesic submanifolds, see \cite{wilking-notes}.
\end{remark}

\subsection{Strongly nonnegative curvature}
\label{ssec:strnneg}
Define the invariant $4$-form $\omega_0=\omega_0(\vec s)$ on each Wallach flag manifold $W^6$, $W^{12}$ and $W^{24}$ by setting:
\begin{equation}\label{eq:ar}
a_r:=\frac{(s_{r+1} - s_{r+2})^2 - s_r^2}{2 s_r}, \qquad 1\leq r\leq 3,
\end{equation}
on \eqref{eq:4formw6}, \eqref{eq:4formw12} and \eqref{eq:4formw24}, and $b_r:=0$ for $W^{12}$. Note that these values of $a_r$ are precisely those for which the third blocks $R_{W^\bullet}^3(\vec s,\omega_0)$ vanish, see \eqref{eq:R3w6}, \eqref{eq:R3w12} and \eqref{eq:R3w24} and Remark~\ref{rem:secondblocks}. In particular, $\omega_0$ depends linearly on the curvature operator of $\g_{\vec s}$.
We now demonstrate that $\omega_0$ is the ``best'' $4$-form for the purpose of verifying whether $\g_{\vec s}$ has strongly nonnegative curvature; in particular, proving Theorem~\ref{strongnonneg}.

\begin{proposition}\label{prop:strnneg}
Let $\g_{\vec s}$ be a homogeneous metric on $W^6$, $W^{12}$, or $W^{24}$, with curvature operator $R$, and let $\omega_0$ be defined as above. The following are equivalent:
\begin{itemize}
\item[(1)] $\g_{\vec s}$ has strongly nonnegative curvature, 
\item[(2)] $\g_{\vec s}$ has $\sec\geq0$,
\item[(3)] $\prs\geq 0$, for $r=1,2,3$, 
\item[(4)] $R+\omega_0$ is a positive-semidefinite operator. 
\end{itemize}
\end{proposition}

\begin{proof}
The implications (1) $\Rightarrow$ (2) and (4) $\Rightarrow$ (1) are trivial (see Section~\ref{sec:strongpos}), and the equivalence (2) $\Leftrightarrow$ (3) follows from Proposition~\ref{prop:modulisec}. We conclude by proving the crucial implication (3) $\Rightarrow$ (4).

By the totally geodesic embeddings \eqref{eq:totgeod}, it suffices to prove (3) $\Rightarrow$ (4) in the case of $(W^{24},\g_{\vec s})$, as modified curvature operators of $(W^6,\g_{\vec s})$ and $(W^{12},\g_{\vec s})$ are restrictions of modified curvature operators of $(W^{24},\g_{\vec s})$.
It is a direct verification that $\widehat{R}^3_{W^{24}}(\vec s, \omega_0)=0$, cf.\ \eqref{eq:R3w24} and \eqref{eq:ar}.
Furthermore, a simple calculation shows
\begin{equation}\label{eq:secnonnegR2w24}
\widehat{R}^2_{W^{24}}(\vec s, \omega_0)=\diag\left(\tfrac{2}{s_r}\prs,\; 1\leq r\leq3 \right)\!,
\end{equation}
which is hence positive-semidefinite by condition (3), see \eqref{eq:R2w24} and \eqref{eq:pr}. 
If all $s_r$ are equal, direct inspection shows that the first block $\widehat{R}^1_{W^{24}}(\vec s, \omega_0)$ is positive-semidefinite. Otherwise, it follows from \eqref{eq:R1w24} and \eqref{eq:ar} that the $2\times 2$ principal minors of $\widehat{R}^1_{W^{24}}(\vec s, \omega_0)$ are:
\begin{equation*}
4s(s_{r+1}-s_{r+2})^2, \quad 1\leq r\leq 3,
\end{equation*}
hence nonnegative. Moreover, $\widehat{R}^1_{W^{24}}(\vec s, \omega_0)v=0$ where
\begin{equation}\label{eq:specialv}
v:=\left(\frac{s_2-s_3}{s_1},\; \frac{s_3-s_1}{s_2},\; \frac{s_1-s_2}{s_3}\right)^\mathtt t\!.
\end{equation}
Since its diagonal entries are positive, Sylvester's criterion implies that $\widehat{R}^1_{W^{24}}(\vec s, \omega_0)$ is positive-semidefinite. Therefore, also $R+\omega_0$ is positive-semidefinite, that is, (4) holds.
\end{proof}

\begin{remark}\label{rem:secondblocks}
Note that the positive-semidefiniteness of each one of the second blocks $\widehat{R}^2_{W^{6}}(\vec s, \omega_0)$, $\widehat{R}^2_{W^{12}}(\vec s, \omega_0)$ and $\widehat{R}^2_{W^{24}}(\vec s, \omega_0)$ is equivalent to the conditions $\prs\geq0$ for the homogeneous metric $\g_{\vec s}$ to have $\sec\geq0$. This can be easily verified through computations analogous to \eqref{eq:secnonnegR2w24}. In fact, the definition \eqref{eq:ar} of $\omega_0$ was originally motivated by the observation that positive-semidefiniteness of the second and third blocks $\widehat{R}^2_{W^\bullet}(\vec s, \omega)$ and $\widehat{R}^3_{W^\bullet}(\vec s, \omega)$ is equivalent to $a_r\in I_{r,\vec s}$, for certain intervals $I_{r,\vec s}$ that are nonempty if and only if $\prs\geq0$. The choice \eqref{eq:ar} corresponds to setting $a_r$ equal to the left endpoint of $I_{r,\vec s}$, which conveniently implies that $\widehat{R}^3_{W^\bullet}(\vec s, \omega_0)=0$.
Finally, note that this observation yields a proof that (1) $\Rightarrow$ (3) which is independent of Proposition~\ref{prop:modulisec}.
\end{remark}


\subsection{Strongly positive curvature}
\label{ssec:strpos}
We now determine the moduli spaces of homogeneous metrics with strongly positive curvature on the Wallach flag manifolds; in particular, proving Theorem~\ref{mainthm}.
This is done studying first-order perturbations of the positive-semidefinite matrices $\widehat{R}_{W^\bullet}(\vec s, \omega_0)$ in Proposition~\ref{prop:strnneg}; as finding a $4$-form $\omega'$ whose restriction to $\ker\widehat{R}_{W^\bullet}(\vec s, \omega_0)$ is positive-definite implies that $\widehat{R}_{W^\bullet}(\vec s, \omega_0+\varepsilon\,\omega')$ is positive-definite for small $\varepsilon>0$, see \cite[Lemma 4.1]{strongpos}.


\begin{proposition}\label{prop:strpos}
The homogeneous metric $\g_{\vec s}$ has strongly positive curvature if and only if $\prs>0$, $r=1,2,3$ and
\begin{itemize}
\item[(1)] $s_r$ are \emph{not all equal}, in the case of $W^6$ and $W^{12}$;
\item[(2)] $s_r$ are \emph{pairwise distinct}, in the case of $W^{24}$.
\end{itemize}
\end{proposition}

\begin{proof}
Suppose $(W^\bullet,\g_{\vec s})$ has strongly positive curvature, hence also $\sec>0$. From Proposition~\ref{prop:modulisec}, it follows that $\prs>0$, $r=1,2,3$ and $s_r$ are not all equal.\footnote{
This implication can be proved independently of Proposition~\ref{prop:modulisec}. Indeed, positive-definiteness of the second and third blocks $\widehat{R}^2_{W^\bullet}(\vec s, \omega)$ and $\widehat{R}^3_{W^\bullet}(\vec s, \omega)$ directly implies that $\prs>0$, $r=1,2,3$, cf.\ Remark~\ref{rem:secondblocks}; and simultaneous positive-definiteness of the first block $\widehat{R}^1_{W^\bullet}(\vec s, \omega)$ is impossible if the $s_r$ are all equal, since its determinant is negative.}
Furthermore, in the case of $W^{24}$, if any two $s_r$ coincide, then there exists a Riemannian submersion $(W^{24},\g_{\vec s})\to(\Ca P^2,\check\g)$, where $\check\g$ is a multiple of the standard metric (see Section~\ref{sec:flags}). Since Riemannian submersions preserve strongly positive curvature \cite[Thm.\ A]{strongpos}, it would follow that $(\Ca P^2,\check{\g})$ has strongly positive curvature, contradicting \cite[Thm.\ C]{strongpos}. 

For the converse, the case of $(W^6,\g_{\vec s})$ follows directly from that of $(W^{12},\g_{\vec s})$, due to the totally geodesic embeddings \eqref{eq:totgeod}.
Assume that $\prs>0$ and that the $s_r$ are not all equal.
We now show that $\widehat{R}_{W^{12}}(\vec s, \omega)$, and hence the modified curvature operator 
$R+\omega$, are positive-definite for $\omega$ arbitrarily close to $\omega_0$, defined in \eqref{eq:ar}.
Recall that the third block $\widehat{R}_{W^{12}}^3(\vec s, \omega_0)$ vanishes identically, and the second block $\widehat{R}_{W^{12}}^2(\vec s, \omega_0)$ is positive-definite. As to the first block $\widehat{R}_{W^{12}}^1(\vec s, \omega_0)$, note that
\begin{equation}\label{eq:block1w12}
\det\begin{pmatrix}
4s_{r+1} & -\frac{s}{s_r}+2a_r\\
-\frac{s}{s_r}+2a_r& 4s_{r+2}
\end{pmatrix}=\frac{4s}{s_r^2}(s_{r+1}-s_{r+2})^2,
\end{equation}
see \eqref{eq:R1w12} and \eqref{eq:ar}.
If, on the one hand, $s_r$ are pairwise different, then $\widehat{R}^1_{W^{12}}(\vec s, \omega_0)$ is positive-definite. In this case, for any $a'_r<0$, the $4$-form $\phi'=a'_1\phi_1+a'_2\phi_2+a'_3\phi_3$ becomes positive-definite on the subspace of representatives corresponding to $\ker\widehat{R}_{W^{12}}(\vec s, \omega_0)$. Thus, the first-order perturbation $\widehat{R}_{W^{12}}(\vec s, \omega_0 +\varepsilon \phi')$ is positive-definite for sufficiently small $\varepsilon>0$.
If, on the other hand, $s_r\neq s_{r+1}=s_{r+2}$, then the $2\times 2$ matrix in \eqref{eq:block1w12} has kernel spanned by $w:=(1,1)^\mathtt t$. 
The restriction of
\begin{equation*}
\phi'+\psi'=a'_1\phi_1+a'_2\phi_2+a'_3\phi_3+b'_1\psi_1+b'_2\psi_2+b'_3\psi_3
\end{equation*}
to the corresponding subspace of representatives reduces to multiplication by
\begin{equation*}
w^\mathtt t
\begin{pmatrix}
-b'_{r+1} & 2a'_r\\
2a'_r& b'_{r+2}
\end{pmatrix}
w=4a'_r-b'_{r+1}+b'_{r+2}.
\end{equation*}
Setting $a'_r<0$, $b'_{r}=b'_{r+1}=0$ and $b'_{r+2}>-4a'_r$, the above $\phi'+\psi'$ becomes positive-definite on the subspace of representatives corresponding to $\ker\widehat{R}_{W^{12}}(\vec s, \omega_0)$.
Similarly to the previous case, the first-order perturbation $\widehat{R}_{W^{12}}\big(\vec s, \omega_0+\varepsilon\,(\phi'+\psi')\big)$ is positive-definite for sufficiently small $\varepsilon>0$.
Thus, we conclude that $(W^{12},\g_{\vec s})$ has strongly positive curvature.

Finally, for the converse in the case of $(W^{24},\g_{\vec s})$, assume that $\prs>0$ and that the $s_r$ are pairwise distinct. Once more, the third block $\widehat{R}_{W^{24}}^3(\vec s, \omega_0)$ vanishes identically, and the second block $\widehat{R}_{W^{24}}^2(\vec s, \omega_0)$ is positive-definite.
Furthermore, the first block $\widehat{R}_{W^{24}}^1(\vec s, \omega_0)$ is positive-semidefinite, with kernel spanned by the vector $v$ in \eqref{eq:specialv}. 
The restriction of $\zeta'=a'_1\,\zeta_1+a'_2\,\zeta_2+a'_3\,\zeta_3$ to the corresponding subspace of representatives reduces to multiplication by the scalar
\begin{equation*}
v^\mathtt t
\begin{pmatrix}
0 & -2a'_3 & -2a'_2 \\
-2a'_3 & 0 & -2a'_1 \\
-2a'_2 & -2a'_1 & 0
\end{pmatrix}v=\frac{4}{s_1s_2s_3}\sum_{r=1}^3 \big(s_r(s_{r}-s_{r+1})(s_r-s_{r+2})\big)a'_r.
\end{equation*}
If $s_r$ are pairwise different, the product of the coefficients of $a'_r$ in the above sum is
\begin{equation*}
\prod_{r=1}^3 \big(s_r(s_{r}-s_{r+1})(s_r-s_{r+2})\big)=-s_1s_2s_3(s_1-s_2)^2(s_1-s_3)^2(s_2-s_3)^2<0,
\end{equation*}
hence at least one coefficient is negative. Setting the corresponding $a'_r$ to be sufficiently negative and $a'_{r+1},a'_{r+2}<0$  small, the $4$-form $\zeta'$ becomes positive-definite on the subspace of representatives associated to $\ker\widehat{R}_{W^{24}}(\vec s, \omega_0)$. Thus, the first-order perturbation $\widehat{R}_{W^{24}}(\vec s, \omega_0+\varepsilon\,\zeta')$ is positive-definite for sufficiently small $\varepsilon>0$. Therefore, we conclude that $(W^{24},\g_{\vec s})$ has strongly positive curvature.
\end{proof}

\end{document}